\documentclass{article}
\usepackage{hyperref}
\usepackage{amsmath}
\usepackage{amsfonts}
\usepackage{amsthm}
\usepackage{amssymb}
\usepackage{float}
\usepackage{varioref}
\usepackage[margin=1in]{geometry}
\usepackage{graphicx}
\usepackage{array}
\usepackage{tikz}
\usepackage{tikz-cd}
\usepackage{fancyhdr}

\usepackage[normalem]{ulem}
\usepackage{mathtools}
\usepackage{tipa}
\usepackage{resizegather}
\usepackage{svg}
\usepackage{comment}

\usepackage[backend=biber]{biblatex}
\theoremstyle{plain}
\newtheorem{theorem}{Theorem}

\newtheorem{corollary}[theorem]{Corollary}
\newtheorem{lemma}[theorem]{Lemma}
\newtheorem{conjecture}[theorem]{Conjecture}

\theoremstyle{definition}
\newtheorem{definition}[theorem]{Definition}
\newtheorem{example}[theorem]{Example}

\newcommand{\arc}[1]{{%
  \setbox9=\hbox{#1}%
  \ooalign{\resizebox{\wd9}{\height}{\texttoptiebar{\phantom{A}}}\cr#1}}}
\title{\textbf{Algebraic Properties of Euclidean Geometry with Transcendental Curves}}
\author{Nicole Venner\thanks{The author thanks John Carter for his wonderful advice on this project.}}
\date{January 1, 2023}

\begin{document}
\maketitle

\begin{abstract}
	While geometry with transcendental curves, like the Quadratrix of Hippias and the Spiral of Archimedes, played a significant role in our modern developments of geometry and algebra. The investigation has fallen off in the modern era despite advancements in algebraic tooling. This paper \footnote{Discussion of original results begins with Definition \vref{def:beginningofresults}} gives a description of the fields using modern techniques such as Galois theory while solving an open conjecture in a 1988 paper \cite{gleason_angle_1988} to provide an answer to if these curves can solve the problem of doubling the cube.
\end{abstract}
\begin{comment}
	Recently extensions of classical euclidean geometry with paper folding have garnered significant interest in the mathematical community. Although these approaches do dramatically expand the domain of solvable problems, several important problems remain unconstructive in this approach. This paper takes an alternative approach, looking to the first historical attempt at extending euclidean geometry, the Quadratrix of Hippias to generate a field of numbers that can hopefully solve those problems.
\begin{enumerate}
	\item Fix Sympy Citation
    \item 
\end{enumerate}
\end{comment}

\section*{A Quick Review of Geometries}
Over 2000 years ago, the Ancient Greeks developed a system of mathematics in response to the discovery that some numbers such as $\sqrt{2}$ were not constructible as proportions of integers. 
\begin{definition}[Euclidean Geometry]
	A point $\alpha$ is \textbf{constructible} if and only if $\alpha$ can be constructed as a point using an unmarked ruler and compass. More specifically those tools allow you to:
	\begin{enumerate}
		\item Draw a line through any two points.
		\item Draw a circle centered at one point while passing through another point.
		\item For any 2 curves (either a line or a circle), you can draw points wherever they intersect.
	\end{enumerate}  
\end{definition}
An extremely important step in the understanding of problems in euclidean geometry is the invention of thinking of points on the plane as pairs of real numbers and more recently as complex numbers $a+bi$. Thinking of individual points as numbers allow you to prove some very interesting theorems:
\begin{theorem}
	Given any 2 constructible points $a, b\in \mathcal{C}$ represented as complex numbers. Then the points with complex representations of $a+b$, $a-b$, $ab$ and $\frac{a}{b}$ are constructible.
\end{theorem}
\begin{proof}
	For any constructible numbers $a,b \in C$ that   $0=a-a$,and $1=a/a$ and $a+b=a-((a-a)-b)$ and $a\cdot b = \frac{a}{\frac{b}{b}\div b}$. You can subtract and divide complex numbers using ruler and compass constructions as shown in Figure \ref{fig:complexdiv} and \vref{fig:complexsub}:
	\begin{figure*}[h!]
		\center
		\includegraphics[width=0.5\textwidth]{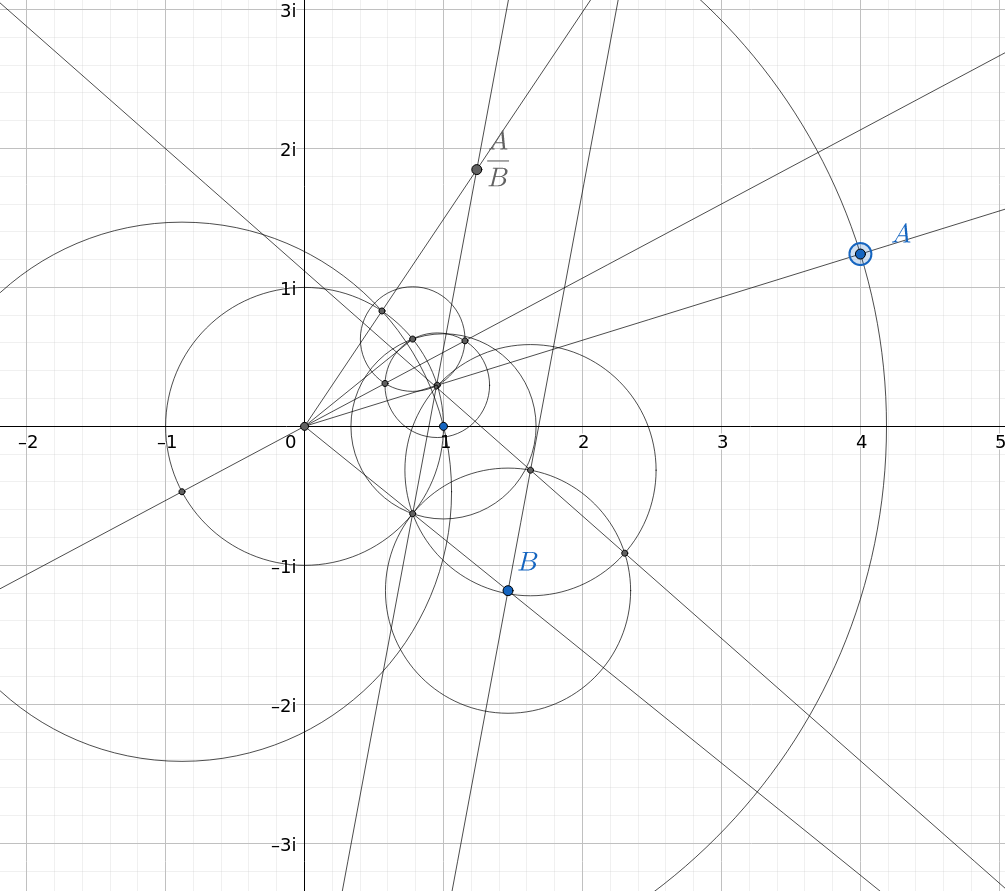}
		\caption[]{Complex Division}
		\label{fig:complexdiv}
	\end{figure*} 
	\begin{figure*}[h!]
		\center
		\includegraphics[width=0.5\textwidth]{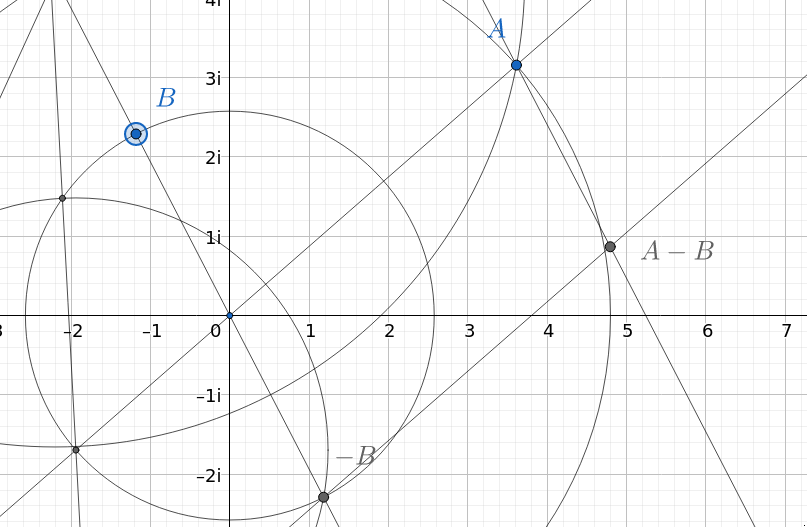}
		\caption[]{Complex Subtraction}
		\label{fig:complexsub}
	\end{figure*}
\end{proof}
\begin{lemma}
	For any numbers $a$ exist geometric constructions for: $\bar{a},\text{Re}(a),\text{Im}(a),|a|$ and $\sqrt{a}$
\end{lemma}
\begin{proof}
	Constructing the $|a|$ is relatively easy since the construction involves just constructing a circle with $a$ as a point on the radius and finding the intersection with the real line. From there you can find $\bar{a}$ since $\bar{a}=\frac{|a|^2}{a}$. This also lets you say construct $\text{Re}(a)=\frac{a+\bar{a}}{2}$ and $\text{Im}(a)=a-\text{Re}(a)$. The construction involving square roots in the complex plane follows from the general square root construction from Euclid combined with an angle bisector.
\end{proof}

As the Ancient Greeks were first developing geometry three classical problems remained unsolved. 
\begin{enumerate}
	\item Construct a square with the same area as a circle.
	\item Construct a cube with double the volume of another cube.
	\item Trisect an angle. (In general arbitrary angle partitions)
	\item Construct a regular-sided heptagon (In general an any-sided regular polygon)
\end{enumerate}
Under the classical rules of euclidean geometry, all three of these problems are impossible and we shall spend the next few sections giving a brief overview of why. The key insight comes from the fact that a solution to each one of these problems must involve constructing certain numbers, namely $\sqrt{\pi}$, $\sqrt[3]{2}$ and $\sin(20^\circ)$. (Importantly $\sin(20)$ happens to be the root of $8x^3 - 6x + \sqrt{3} = 0$)

So what numbers can be constructed with Ancient Greek geometry? Since you are dealing with the geometry of a two-dimensional plane it makes sense to involve numbers that carry the same structure, namely complex numbers. From here you can define

It turns out that these are all the numbers you can construct with classical euclidean geometry, via the following theorem:
\begin{theorem}
	Any Intersections that can be created with classic euclidean geometry (line intersected with line, circle and line, or circle and circle), can be represented using roots of quadratic polynomials
\end{theorem}
\begin{proof}
	Since in the cartesian plane, all circles can be written as $(x-c)^2+(y-d)^2=r^2$ and lines as $y=ax+b$. Then you can algebraically represent the x and y coordinates of intersections as solutions to polynomials, the calculations for a line intersecting a line and a circle intersecting a circle are relatively straightforward, so all that is left is to work out the algebra in the case of a circle intersecting with another circle.  
	% \begin{align*}
	% 	ax+b&=\pm\sqrt{r^2-(x-c)^2}+d\\
	% 	a^2x^2+2a(b-d)x+(b-d)^2&=r^2-(x-c)^2\\
	% 	a^2x^2+2a(b-d)x+(b-d)^2&=r^2-c^2x^2+2xc-c^2\\
	% 	(a^2+c^2)x^2+2(ba-da-c)x+(b-d)^2+c^2-r^2&=0
	% \end{align*}
	\begin{align*}
		\sqrt{r^2-(x-a)^2}+c&=\sqrt{p^2-(x-b)^2}+d\\
	  \sqrt{r^2-(x-a)^2}&=\sqrt{p^2-(x-b)^2}+(d-c)\\
		r^2-(x-a)^2&=\left(\sqrt{p^2-(x-b)^2}+(d-c)\right)^2\\
	  r^2-(x-a)^2&=p^2-(x-b)^2+2(d-c)\sqrt{p^2-(x-b)^2}+(d-c)^2\\
	  \frac{(r^2-p^2)+(x-b)^2-(x-a)^2-(d-c)^2}{2(d-c)}&=\sqrt{p^2-(x-b)^2}\\
	  \frac{(2b-2a)x+(b^2+r^2-p^2-a^2-d^2+2dc-c^2)}{2(d-c)}&=\sqrt{p^2-(x-b)^2}\\
	  \frac{\left((2b-2a)x+(b^2+r^2-p^2-a^2-d^2+2dc-c^2)\right)^2}{4(d-c)^2}&=p^2-(x-b)^2
	\end{align*}
	Expanding out yields:
	\begin{gather*}
	  \frac{2x^2(b-a)^2+4x(b-a)(b^2+r^2-p^2-a^2-d^2+2dc-c^2)+(b^2+r^2-p^2-a^2-d^2+2dc-c^2)^2}{4(d-c)^2}=(p^2-(x-b)^2)\\
	  2x^2(b-a)^2+4x(b-a)(b^2+r^2-p^2-a^2-d^2+2dc-c^2)+(b^2+r^2-p^2-a^2-d^2+2dc-c^2)^2=4(d-c)^2(-x^2+2xb-b^2+p^2)\\
	  2x^2\left((b-a)^2+2(d-c)^2\right)+4x\left((b-a)(b^2+r^2-p^2-a^2-d^2+2dc-c^2)-2b(d-c)^2\right)+(b^2+r^2-p^2-a^2-d^2+2dc-c^2)^2+4(d-c)^2(b^2-p^2)=0
	\end{gather*}
	Since the $x$ is solvable, we can obtain the $y$ value by substituting $x$ into either $ax+b=y$ or $(x-a)^2+(y-b)^2=r^2$ yielding a quadratic as well. 	
\end{proof}
Since any quadratic can be solved using the quadratic formula, we can construct every constructible number, using the operations we have defined so far $(+,-,\times, \div, \sqrt{})$ suffice to construct any constructible number. Having reduced the problem to algebra we can show it just isn't enough to solve those classic geometry problems.
\subsection*{The Foldable Numbers}

Knowing that it is only possible to construct solutions of quadratics, it seems natural to ask if other possible sets of constructions could create more numbers. A recent modern example is the trigonometric numbers: 
\begin{definition}
	A number $\alpha$ is in $\mathcal{O}$ if the point representing the complex number $\alpha$ is constructible using simple reduced\footnote{Every other kind of simple paper fold can be completed using repeated applications of operation 3} origami folds
	\begin{enumerate}
		\item Given two points P1 and P2, we can fold a crease line connecting them. 

		\item Given two lines, we can locate their point of intersection, if it exists. 
		\item Given two points P1 and P2 and two lines L1 and L2, we can, whenever possible, make a fold that places P1 onto L1 and also places P2 onto L2.
	\end{enumerate}
\end{definition}
This may seem like an oversimplification of origami, but helpfully a proof from \textcite{hull_origametry_2020} (originally proved by Lang in 2003) shows us that
\begin{theorem}[Lang]
	If we only allow one fold at a time and assume all our creases are straight lines. Then any folding operation can be described by our three origami operations
\end{theorem}
Using these techniques it becomes possible to create solutions to cubic equations, consider the following construction for the cube root of 2 in Figure \vref{fig:cuberoot2}
\begin{figure*}[h!]
	\center
	\includegraphics[width=0.5\textwidth]{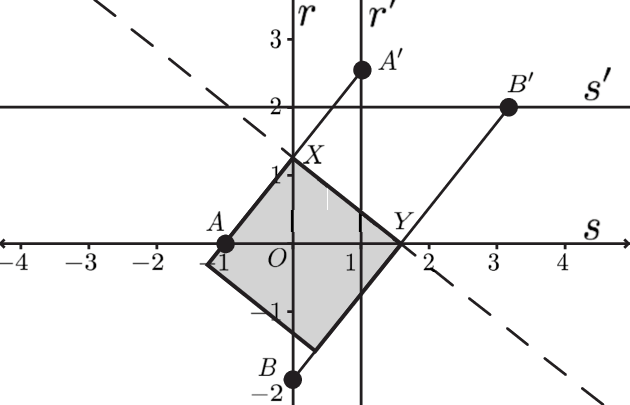}
	\caption{Construction of the cube root of 2 (Hull) (Maybe replace with own?)}
	\label{fig:cuberoot2}
\end{figure*}
This is done by moving the $x$ and $y$ axes to intersect at (1,2) and then using our third rule to map $(-1,0)$ and $(0,-2)$, to the new axes. From there it is possible to use conventional geometry and an argument from similar triangles to argue that the side length of the square is the cube root of 2.

This exact argument can be easily generalized to solving any cubic as first proved by Lil.

\begin{theorem} [Lil]
	The roots of any cubic can be expressed using simple paper folds.
\end{theorem}
The author's first exposure to this was in a video produced by \textcite{polster_why_2019}. Furthermore using a similar argument from algebra it turns out to completely describe the numbers you can make with simple folds in origami. \footnote{There do exist extensions of origami involving multi-folds that create multiple creases at once. While these drastically expand the numbers that can be produced. The author believes there are significant issues in using them for real-life demonstrations or educational purposes. Not only since the constructions require a large amount of technical skill to execute, but inaccuracies in the construction are known to propagate and grow, making developing intuition quite difficult.}
\begin{theorem}
	The intersections created by any origami fold can be described by a cubic polynomial.
\end{theorem}
Proof for this can be found in \textcite{hull_origametry_2020}

\subsection*{Galois Theory and Unconstructible Numbers}
From here since we have 2 notions of geometry that can be modeled as algebraically solving polynomial equations, it remains if it is possible to construct different numbers like $\sqrt[5]{2}$ or transcendental numbers like $\pi$ or $e$.

The original idea that these problems are impossible to construct in either comes from a branch of mathematics called Galois Theory, which we shall briefly stroll-through for the sake of brevity.
\begin{definition}
    A polynomial is irreducible if it cannot be factored.
\end{definition}
\begin{definition}
	A \textbf{Field} $\mathcal{F}$ is a set of numbers that is closed under the operations required for traditional arithmetic (addition, subtraction, multiplication, and division). A Field Extension is a process of extending a field $\mathcal{F}$ with new numbers to form a field $\mathcal{G}$, this extension is often written as $G/F$
\end{definition}
\begin{definition}
	A field $\mathcal{F}$ \textbf{splits} over a polynomial $p(x)$ if every root of $p(x)$ is contained in $\mathcal{F}$. An extension $\mathcal{F}/\mathcal{G}$ is a splitting extension if for any irreducible polynomial in $\mathcal{G}$ then either all or none of its roots are in $\mathcal{F}$.
\end{definition}
While this is a helpful definition, it might seem hard to come up with examples where this is true.
\begin{theorem}
	For any irreducible polynomial $p(x$ with coefficients $\mathcal{F}$, then the extension smallest field containing all roots $\mathcal{F}(x_1,x_2,\dots)/\mathcal{F}$ is splitting.
\end{theorem}
this can be shown to not be the case.

Some examples and a short lemma might be helpful to understand the concept.
\begin{example}
	Our first example is going to be the smallest field containing both the rationals and $\sqrt{5}$, written as $\mathbb{Q}(\sqrt{5})$. It is reasonable to ask what the elements of this field look like, and it turns out it is completely describable using 2 rational numbers like so: $\mathbb{Q}(\sqrt{5})=\left\{a+b\sqrt{5}| a,b \in \mathbb{Q}\right\}$. These form a field since you can prove they are closed under subtraction and division like so:
	\begin{align*}
		(a+b\sqrt{5})-(c+d\sqrt{5})&=(a-b)+(b-d)\sqrt{5}\\
		\frac{a+b\sqrt{5}}{c+d\sqrt{5}}&=\left(\frac{ac-5bd}{c^2-5d^2}\right)+\left(\frac{bc-ad}{c^2-5d^2}\right)\sqrt{5}
	\end{align*}
	We can also see that this field is the splitting field for the polynomial $x^2-5$ since both of its roots of $\sqrt{5}$ and $-\sqrt{5}$ are contained in it.
\end{example}
\begin{lemma}
	If a field $\mathcal{F}$ is composed of a tower of fields $\mathbb{Q} \subset \mathcal{F}_1\subset \mathcal{F}_2\subset \mathcal{F}_3 \subset \dots$, where each extension is splitting. Then $\mathcal{F}/\mathbb{Q}$ is splitting.
	\label{lem:splittower}
\end{lemma}
As proved in \cite{dummit_abstract_2009}.
\begin{definition}
	A \textbf{field automorphism} is an invertible function on a field such that.
	\begin{equation*}
		f(a-b)=f(a)-f(b) \qquad f\left(\frac{a}{b}\right)=\frac{f(a)}{f(b)}
	\end{equation*}
	Based on those properties we can conclude that $f(a)\cdot f(b)=f(ab)$ and $f(a)+f(b)=f(a+b)$. And for any rational $q$ then $f(q)=q$.
\end{definition}
Through the power of field automorphisms, we can prove some useful results:
\begin{lemma}
	If $f$ is a field automorphism defined on a splitting field $\mathcal{F}$ defined by a polynomial with rational coefficients $p(x)$, for any root $z$, then $f(z)$ is also a root.
\end{lemma}
\begin{proof}
	If $z$ is a root of $p(x)$, then we know by definition that
	\begin{equation*}
		0=a_n z^n+a_{n-1}x^{n-1} + \dots + a_1 x^n +a_0
	\end{equation*}
	We can now apply $f$ to both sides of this equation, and since by the properties of a field automorphism we know that you can distribute across integer powers $f(a^n)=f(a)^n$ (proven by using induction with multiplication) we can finish the proof.
	\begin{align*}
		f(0)&=f \left(a_n z^n+a_{n-1}z^{n-1} + \dots + a_1 z +a_0\right)\\
		0&=f(a_n z^n)+f(a_{n-1}z^{n-1}) + \dots + f(a_1 z) +f(a_0)\\
		0&=a_n f(z^n)+a_{n-1}f(z^{n-1}) + \dots + a_1 f(z) +a_0\\
		0&=a_n f(z)^n+a_{n-1}f(z)^{n-1} + \dots + a_1 f(z) +a_0\\
	\end{align*}
	Thus $f(z)$ is a root of $p(x)$. (This applies generally to polynomials with irrational coefficients if the automorphism $f$ keeps the coefficients fixed in place.)
\end{proof}
\begin{corollary}
	If $\mathcal{F}$ is a splitting field of a polynomial $p(x)$ with rational coefficients, for any root $z$ of $p(x)$, then both the conjugate $\bar{z}$ as well as the real, imaginary parts $\Re(z)$ and $i \cdot \Im(z)$ are in $\mathcal{F}$,
\end{corollary}
\begin{proof}
	Consider that the field automorphism $f(a+bi)=a-bi$. We can prove that this is a field automorphism with some basic algebra
	\begin{align*}
		f(a+bi)-f(c+di)=a-bi-c+di=(a-c)-(b-d)i=f(a-c+(b-d)i)
	\end{align*}
	Furthermore
	\begin{align*}
		\frac{f(a+bi)}{f(c+di)}&=\frac{a-bi}{c-di}\\
		&=\frac{(a-bi)(c+di)}{c^2+d^2}\\
		&=\frac{ac+bd}{c^2+d^2}+i \frac{ad-bc}{c^2+d^2}\\
		&=f\left(\frac{ac+bd}{c^2+d^2}+i \frac{bc-ad}{c^2+d^2}\right)\\
		&=f\left(\frac{a+bi}{c+di}\right)
	\end{align*}
	
	any root $z=a+bi$ of a polynomial with rational coefficients, then consider the 
\end{proof}
Now 
\begin{theorem}
	Every field and field extension can be put into correspondence with a group, often written as $Gal(\mathcal{F})$ or $Gal(\mathcal{F}/\mathcal{G})$. Defined to be the group formed by all field automorphisms on your field. (The Galois Group of a field extension $\mathcal{F}/\mathcal{G}$ are the field automorphisms of $\mathcal{G}$ that keep the values of $\mathcal{F}$ unchanged.)
\end{theorem}
A true understanding of galois theory is going to require a fair bit of understanding of the properties of groups. But even then an example might be helpful:
\begin{example}
	Let $\mathcal{F}$ be the smallest field where $x^3-2$ splits over the rationals. Since $\mathcal{F}$ is the smallest field that contains all the roots of $x^3-2$ we can say $\mathcal{F}=\mathbb{Q}(\sqrt[3]{2},\omega\sqrt[3]{2},\omega^2\sqrt[3]{2})$, where $\omega^3=1$. It also does have a representation using rational numbers like so:
	\begin{equation*}
		\mathcal{F}=\left\{a+b\cdot 2^{1/3}+c\cdot 2^{2/3}+d\cdot\omega+e\cdot\omega 2^{1/3}+f\cdot \omega 2^{2/3} \quad | \quad a,b,c,d,e,f \in \mathbb{Q}\right\}
	\end{equation*}
	Since $\mathcal{F}$ is a splitting extension over the rationals it is relatively easy to calculate its Galois group. Since the rationals (or your preferred base field) must remain fixed by any field automorphism, by the properties of the automorphism for any root $z$, then $f(z)$ must also be a root. Likewise, since every element of your field can be written as an element of the base field $p_n$, times a product of roots to some polynomial $\alpha$, applying the permutation yields:
	\begin{equation*}
		f\left(\sum_{i=1}^k a_i \cdot \alpha_1^{n_{1_i}}\cdot \alpha_2^{n_{2_i}}\cdot \alpha_3^{n_{3_i}}\cdot \dots\right)=\sum_{i=1}^k a_i \cdot f(\alpha_1)^{n_{1_i}}\cdot f(\alpha_2)^{n_{2_i}}\cdot f(\alpha_3)^{n_{3_i}}\cdot \dots
	\end{equation*}
	Thus the field automorphisms are given by any permutation of the roots of the polynomial while keeping all of their algebraic properties intact.
	
	For our specific polynomial, the only algebraic properties the roots need to respect are:
	\begin{gather*}
		x_0^3=x_1^3=x_2^3=x_0\cdot x_1\cdot x_2=2\\
	\end{gather*}
	Thus any permutation leads to a valid automorphism. Showing that the Galois group of $x^3-2$ is the symmetric group $S_3$.
\end{example}
\begin{definition}
	The \textbf{order} of any (Galois) group $G$ written as $|G|$ is defined to be the number of elements (field automorphisms) in the group.
\end{definition}
\begin{theorem}
	For any fields $F\subset G$ with a splitting field extension $G/F$. Then
	\begin{equation*}
		|Gal(G/F)| \cdot |Gal(F)|=|Gal(G)|
	\end{equation*}
\end{theorem}
Furthermore
\begin{theorem}
	For any irreducible polynomial $p$ of degree $n$, the splitting field $F$ containing all the roots $F(p_1,p_2,\dots)$ then the order of $Gal(F(p_1,p_2,\dots)/F)$ is divisible by $n$ and divides $n!$.
\end{theorem}
For interested readers, the proofs for these statements are a combination and simplification of several proofs in the Galois theory chapter of \textcite{dummit_abstract_2009}. Using these 2 theorems we can show that the three problems are impossible. First starting by saying that
\begin{theorem}
	Every constructible number $c$ is contained in a field with a Galois group with an order of $2^n$ for some natural number $n$.
\end{theorem}
\begin{proof}
	Since every constructible number is created using a finite number of ruler and compass constructions, you can interpret it as extending the base field of the rationals using a finite number of solutions to quadratic polynomials. Combining the previous 2 theorems, we get that the order of the final field containing our number $c$ must be contained in a field with degree $2^n$.
\end{proof}
Likewise following the same argument one can see that

But notice that classic problems of doubling the cube and trisecting the angle involve extensions of irreducible cubic polynomials, $x^3-2$ and $8x^3 - 6x + \sqrt{3}$. Thus any 'nice' field containing them must be divisible by 3\footnote{If you calculate them out it turns out that the first polynomial has a Galois group of order 6 and the second has order 3.} But since every number constructible with classic euclidean geometry has a degree with no factors of 3. 

We can further see that the field of foldable numbers can solve these since:

From here we have one final result 
\begin{corollary}
	For any number $o \in \mathcal{O}_1$ then the smallest field containing $o$ is of degree $2^n \cdot 3^k$ for any integers $n$ and $k$
\end{corollary}
Since it is possible to describe doubling the cube and trisecting the angle as the solutions of the cubics $x^3-a$ and $3x^3-2x+c$ for $c<1$ and any $a$. However, some numbers are still out of reach, arbitrary angle division and arbitrary regular polygons are still unconstructible as well as the elusive square of the circle.

% It is also possible to prove with galois theory that not every polynomial has roots that can be expressed using $n$th roots, even though every polynomial with a degree less than 5 can be expressed in this way. Leading to the following 2 fields:
% \begin{definition}[The Solvable Numbers: $\mathcal{S}$]
% 	Numbers created with $n$th roots of numbers plus the standard field operations.
% \end{definition}

\section*{The Quadratrix of Hippias}
In Ancient Greece, the first partial solution to the angle trisection problem comes from Hippias where he imagines a new curve in the plane drawn with a compass-like construction. The definition of the curve can be given as so. 
\begin{figure}[h!]
	\center
	\includegraphics*[width=.4\textwidth]{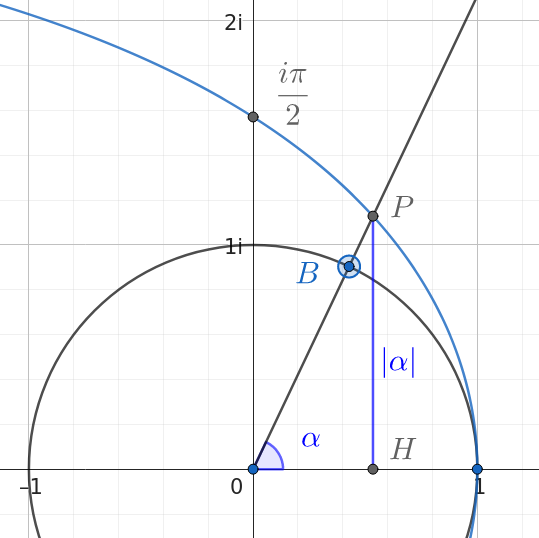}
	\caption{Definition of the Quadratrix}
	\label{fig:quaddef}
\end{figure}
\begin{definition}
	Given a circle with center $A$ and point on the radius $B$. The quadratrix is defined so that for any point on the circle $X$. Then the point on the quadratrix $Z$ on the line $\overline{AX}$ is defined so that the arc-length of $\arc{BX}$ is equal to the distance from the line $\overline{AB}$ As seen in Figure \vref{fig:quaddef}
\end{definition}

\begin{example}(Construction of the number $\pi/7$)
	Using the definition of the quadratrix, we can try to find the points where it intersects the imaginary axis, since the imaginary line is angled at $\pi/2$ radians from the real line, then by \vref{fig:quaddef}, it must intersect at the points $\pm \frac{i\pi}{2}$. Since the number $i\pi/2$ is constructible, we can multiply by the constructible number $\frac{-2i}{7}$ to construct the point $\frac{\pi}{7}$.
\end{example}
\begin{example}(Partition of an angle into 5 pieces)
	\begin{figure}[h!]
		\center
		\includegraphics*[width=.4\textwidth]{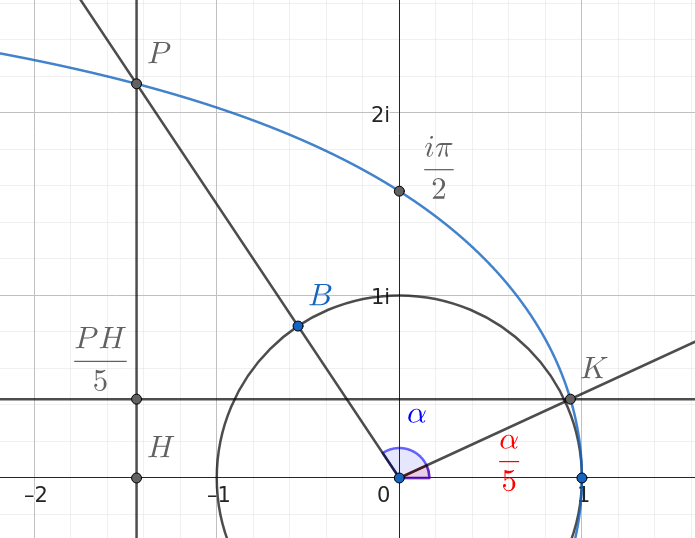}
		\caption{Partition of an angle into 5 parts.}
		\label{fig:quadquint}
	\end{figure}
\end{example}
Hippias also proved that this curve can partition an arbitrary angle, and can therefore create any regular-sided polygon. 130 years later the mathematician Dinostratus proved that it is also possible to square the circle as well. (Both techniques will be proved later by saying that the field containing quadratrix is also inside) But from here to the best of the author's knowledge no one has done a further dive into the quadratrix aside from reiterating the results from over 2 millennia ago.

% In an attempt to keep things clear on my part here is a poset diagram with all the results I have proven so far marked with $\Rightarrow$, existing or trivial results marked with $\rightarrow$, and open questions marked with $\dashrightarrow$

% https://tikzcd.yichuanshen.de/#N4Igdg9gJgpgziAXAbVABwnAlgFyxMJZAJgBpiBdUkANwEMAbAVxiRAB12BbOnACwDGjYAGUAviDGl0mXPkIoyAZiq1GLNpx78ARjuABFMQD1gAQR0SpM7HgJEl5VfWatEHbr0HCA8mID6xJLSIBi28kRkAIzO6m4e2nx6wADCJgBUALxaXkIMwH7+nFhgAGY4AJ7BNnL2iqQADLGump66+mnVobJ2CiSkACzNGu45-HmpViFhtX1kAKzD8WNJ+kZdM71EUU7ULiMJucIAEgFRGz0RKDsqe3GtiRMAKmcX4XXIOzF3LaNt3vkXm9Zg5SLc1L9DuNfK9rN13n0GrsIQcVs8AkExKoYFAAObwIigUoAJwgXCQSJAOAgSCicJJZKQjipNMQDXppPJiGZ1KQxA5jMQA2ovO5Aq5OxZSHm4ulItZADZZYgAOzypAADmoDBK8SgECYOgYrGofBgdCgSDATAYDBFdCwDDYkDArGVAE51artbq2PqcDgcV0GVzhVLEJ6QDrXWw4BAdZblWrw5To3qDUaTSAzRarTa7VSHU73C63SEQ0hk6LmWm-RAA0HlWRw1F2eXObTKaKZe3BVqWz6Y+59YbjSBTebLYhrbb7Y7nQQy0SO2yva3lVFJaL+7XhxmxxPc9P83Pi+BF8GV5HRZvB+nR1mc1OZwWcEWFzGN1vWcndyB-YGiYUGIQA

\section*{Algebraic Definition of $\mathcal{T}_1$}
As we have just seen, the majority of the power of the quadratrix in classic constructions comes from its ability to convert angles into segment lengths and segment lengths into angles, thus let us consider a field generated by those 2 constructions:
\begin{definition}
	The field of ''Perpendicular Trigonometric'' Numbers $\mathcal{T}_1$ is defined to be the smallest field satisfying
	\begin{enumerate}
		\item The set of rational numbers $\mathbb{Q}$ is contained in $\mathcal{T}_1$ (Lemma \ref{lem:t1constructible} lets us show that every constructible number is in $\mathcal{T}_1$)
		\item The field is closed under the operation defined in Definition \ref{def:seg2arc}
		\item The field is closed under the operation defined in Definition \ref{def:arc2seg}
	\end{enumerate}
	\label{def:beginningofresults}
\end{definition}
\begin{definition}[Arc $\to$ Segment: \textbf{T1}]
	For any arc $\arc{AB}$ it is possible to construct a segment $\overline{CD}$ with length equal to the arc-length of $\arc{AB}$
	\label{def:arc2seg}
\end{definition}
\begin{definition}[Segment $\to$ Arc: \textbf{T2}]
	For any segment $\overline{CD}$ it is possible to construct a arc $\arc{AB}$ with arc-length equal to the length of $\overline{CD}$
	\label{def:seg2arc}
\end{definition}
It is quite apparent that this field is already solving ancient problems from antiquity as since it is easy to construct a segment of length $\pi$, you can easily square the circle (Also showing that $\mathcal{C}(\pi)\subset \mathcal{T}_1$).  It goes significantly deeper than that as the following theorem should show 
 
\begin{theorem}\label{thm:exp}
	For any real trigonometric $\alpha$, then $\alpha$ is trigonometric if and only if $e^{i\alpha}$ is trigonometric. (Alternatively stated, for any $|\beta|=1$ is trigonometric if and only if $\ln(\beta)$ is trigonometric)
\end{theorem}
\begin{figure}[h!]
	\center
	\includegraphics[width=0.5\textwidth]{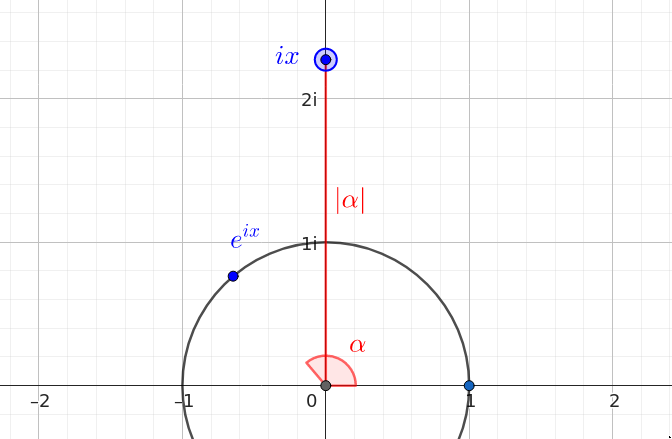}
	\caption[]{A Geometric View of Complex Exponentiation}
	\label{fig:exp}
\end{figure}
\begin{proof}
	If a real number $\alpha$ is constructible then via \textbf{T2} we can wrap the segment from $0$ to $\alpha$ around the unit circle with an endpoint of $\cos(\alpha)+i\sin(\alpha)=e^{i\alpha}$. Likewise for any trigonometric point on the unit circle is expressible as $\cos(\alpha) + i \sin(\alpha)=e^{i\alpha} $then we can use \textbf{T1} to unwind that to a segment length, thereby constructing $\alpha$. (This can also be expressed by taking the natural log of any point on the unit circle in the complex plane.)
\end{proof}
\begin{theorem}[Trigonometric Closure]\label{thm:trig}
	Any real number $\alpha$ is trigonometric if and only if $\sin (\alpha)$, $\cos (\alpha)$, $\tan (\alpha)$, are trigonometric.
\end{theorem}
\begin{proof}[Proof $\Rightarrow$]
	Its notable that if $\alpha$ is real and trigonometric then by theorem \vref{thm:exp}, then $e^{i\alpha}$ is trigonometric and since:
	\begin{align*}
		\sin(x)&=\frac{e^{ix}-e^{-ix}}{2i}\\
		\cos(x)&=\frac{e^{ix}+e^{-ix}}{2}\\
		\tan(x)&=-i\frac{e^{ix}-e^{-ix}}{e^{ix}+e^{-ix}}
	\end{align*}
	Every number of this type is therefore constructible.
\end{proof}
\begin{proof}[Proof $\Leftarrow$]
	Using the functions above, it is possible to derive the following formulas using the socks and shoes rule ie. $(f\circ g)^{-1}=g^{-1} \circ f^{-1}$.
	\begin{align*}
		\sin^{-1}(x)&=-i \ln \left(\sqrt{1-x^2}+ix\right)\\
		\cos^{-1}(x)&= \frac{\pi}{2}+i \ln \left(\sqrt{1-x^2}+ix\right)\\
		\tan^{-1}(x)&= \frac{i}{2} \ln \left(\frac{1-ix}{1+ix}\right)
	\end{align*}
	Where the inverse trig functions are defined for real values ($[0,1]$ for $\sin^{-1}$ and $\cos^{-1}$, and the entire real line for $\tan^{-1}$), then the values that are given to $\ln(x)$ are all of magnitude 1.
\end{proof}

\section*{Geometric Definitions of $\mathcal{T}_1$}
So how does this curve relate to the quadratrix and other transcendental curves throughout history that have been used to solve similar problems? While it might be tempting to look at the fields generated by creating a "quadratrix" compass, or "trigonometric" compass that can draw these shapes and think about what intersections you might get, however trying this method can easily lead to intractable problems since by definition of your field is going to include solutions to equations like this:
\begin{align*}
	\alpha x = \sin(x)
\end{align*}
Since it must include the intersections of a sloped line and the standard trigonometric curve. Based on that it makes sense to examine subfields with fewer allowed constructions.

\subsection*{Trigonometric Curves}
\begin{definition}
	Consider the field of numbers generated by allowing intersections with the curves $x \mapsto (x,\sin(x))$ or $x \mapsto (x, \cos(x))$, then consider the subfield:
	\begin{enumerate}
		\item All the standard constructions in classic euclidean geometry, (ie, creating lines and circles, and finding intersections of any two lines, two circles, or a line and a circle)
		\item Given the curve $x \mapsto (x,\sin(x))$ or $x \mapsto (x, \cos(x))$ when graphed on $\mathbb{R}^2$, you can construct the intersection of the curves, and any vertical line.
		\item Given the curve $x \mapsto (x,\sin(x))$ or $x \mapsto (x, \cos(x))$ when graphed on $\mathbb{R}^2$, you can construct the intersection of the curves, and any horizontal line.
	\end{enumerate}
	\label{def:trigt1field}
\end{definition}
\begin{theorem}
	The field described in Definition \vref{def:trigt1field} is equal to $\mathcal{T}_1$.
\end{theorem}
\begin{proof}
	The first direction of proving that $\mathcal{T}_1$ contains the construction amounts to showing that all the intersections with vertical lines are constructible since this happens to be the same as evaluating the function, and horizontal lines are the same as finding function inverses. Thus applying Theorem \vref{thm:trig} deals with both halves of our equivalence. Showing every number constructible using classic Euclidean constructions will wait until lemma \vref{lem:t1constructible}
\end{proof}
\subsection*{Quadratrix of Hippias}
\begin{definition}
	Consider a subfield of numbers constructible with the quadratrix, generated using these constructions:
	\begin{enumerate}
		\item All the standard constructions in classic euclidean geometry, (ie, creating lines and circles, and finding intersections of any two lines, two circles, or a line and a circle)
		\item Given the standard quadratrix as constructed in Figure \vref{fig:quaddef}, you can construct the intersection of the quadratrix and any line passing through the origin.
		\item Given the standard quadratrix as constructed in Figure \vref{fig:quaddef}, you can construct the intersection of the quadratrix and any horizontal line.
	\end{enumerate}
	\label{def:quadt1field}
\end{definition}
\begin{theorem}
	The field described in Definition \vref{def:quadt1field} is equal to $\mathcal{T}_1$.
\end{theorem}
\begin{proof}
	To proceed with this proof we shall show that every number in each field is present in the other
	$(Quad \Rightarrow \mathcal{T}_1)$. Notice that given any angle measure $\alpha$, you can construct a ray at the origin where the angle formed with the real axis is $\alpha$. Then by definition, the distance between the intersection and the real axis is $\alpha$. Likewise, creating a horizontal line given by $\{x+i\cdot \alpha \quad x \in \mathbb{R}\}$. Then find the intersection with the quadratrix and draw the line to the origin creating an angle with the real axis of value $\alpha$. Showing every number constructible using classic Euclidean constructions will wait until lemma \vref{lem:t1constructible}

	$(\mathcal{T}_1 \Rightarrow Quad )$. Consider a line drawn through the origin with angle $\alpha$, by definition, the point on the quadratrix can be written as $|\alpha|+ i \cdot |\alpha|\cdot \cot(\alpha)$. Since the point is expressible using trigonometric numbers, it does not extend the field further. Likewise, taking a horizontal line, by the first property we can turn its distance from the real axis into an angle value and see where that line intersects 
	
\end{proof}
\subsection*{Archimedean Spiral}
\begin{theorem}
	The field $\mathcal{T}_1$ is equal to the field of constructible numbers with the additional ability to construct intersections with the Archimedean Spiral with lines passing through the origin and circles centered at the origin.
\end{theorem}

\begin{figure}[h!]
		\center
		\includegraphics[width=0.5\textwidth]{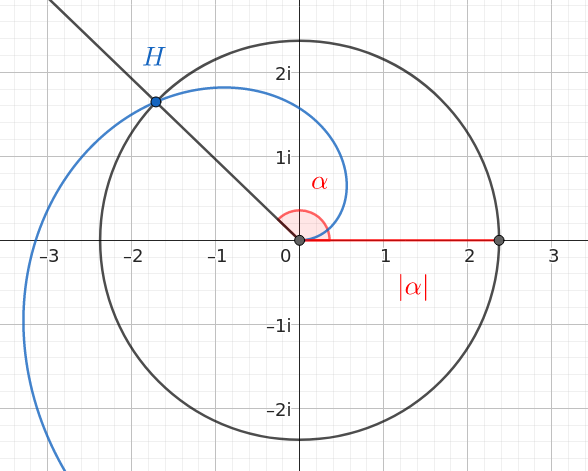}
		\caption[]{The Spiral of Archemedies}
		\label{fig:spiraldef}
\end{figure}

\begin{definition}
	Consider a subfield of numbers constructible with the Spiral of Archimedes, generated using these constructions:
	\begin{enumerate}
		\item All the standard constructions in classic euclidean geometry, (ie, creating lines and circles, and finding intersections of any two lines, two circles, or a line and a circle)
		\item Given the standard Spiral of Archimedes as constructed in Figure \vref{fig:spiraldef}, you can construct the intersection of the spiral and any line passing through the origin.
		\item Given the standard Spiral of Archimedes as constructed in Figure \vref{fig:spiraldef}, you can construct the intersection of the spiral and any circle centered at the origin.
	\end{enumerate}
	\label{def:spiralt1field}
\end{definition}
\begin{theorem}
	The field described in Definition \vref{def:spiralt1field} is equal to $\mathcal{T}_1$.
\end{theorem}
\begin{proof}
	To proceed with this proof we shall show that every number in each field is present in the other

	$(Archm \Rightarrow \mathcal{T}_1)$.  Notice that if you have any angle, you can create a ray from the origin with an angle $\alpha$ the distance from the point of intersection on the spiral and the origin will be $|\alpha|$, thus showing it is constructible. Likewise for any constructible real number $r$, then taking that number we can create a circle of radius $r$ centered at the origin. Drawing a ray from the origin to this point lets us create a ray of $z$ radians from the real axis. Showing every number constructible using classic Euclidean constructions will wait until lemma \vref{lem:t1constructible}
 
	$(\mathcal{T}_1 \Rightarrow Archem)$.   Notice the parameterization of the spiral of Archimedes as $(t\cdot \cos(t),t\cdot \sin(t))$., this lets us algebraically solve for the intersection points of any circle centered at the origin with radius $r$, namely $r\cos(t)+i\cdot r \sin(t)$. Likewise for any ray from the origin with an angle $\theta$ its intersection point will be defined by $\theta \cos(\theta)+ i \cdot \theta + \sin(\theta)$
\end{proof}

\section*{Transcendence of $\mathcal{T}_1$}
Now having established some various formulas for the algebra of $\mathcal{T}_1$ we can continue deriving results first citing an important theorem in transcendental number theory.
\begin{definition}
	The \textbf{transcendence degree} of a field $\mathcal{F}$ is the smallest number of transcendental elements $a_1, \dots, a_n$ such that every element of $\mathcal{F}$ can be expressed as the roots of polynomials with coefficients in $\mathbb{Q}(a_1, \dots, a_n)$
\end{definition}
\begin{theorem}[Lindemann–Weierstrass]
	For any linearly independent set of algebraic numbers $a_1,a_2,\dots,a_n$, then the transcendence degree of $\mathbb{Q}(e^{a_1},e^{a_2} \dots, e^{a_n})$ is $n$
\end{theorem}
A proof for this is available in \textcite{natarajan_pillars_2020} 
\begin{example}
	$\pi$ is transcendental (not the root of any polynomial with rational coefficients)
\end{example}
\begin{proof}
	To show that $\pi$ is not algebraic, we must show that $\pi$ being algebraic leads to a contradiction. If $\pi$ is algebraic, then $i\pi$ is algebraic, and therefore the field $\mathbb{Q}(e^{i\pi})=\mathbb{Q}(-1)=\mathbb{Q}$ has transcendence degree 1. But $\mathbb{Q}$ trivially has a transcendence degree of $0$, so $\pi$ is transcendental.
\end{proof}
\begin{theorem}
	The transcendence degree of $\mathcal{T}_1$ is infinite.
\end{theorem}
\begin{proof}
	The key to this will be the Lindemann–Weierstrass theorem.  Since the set
	\begin{equation*}
		\{i\sqrt{p}\quad  |\quad p \text{ prime}\}
	\end{equation*}
	gives you an infinite set of algebraic numbers that are trigonometric and linearly independent. Via theorem \vref{thm:exp} the number $e^{i \sqrt{p}}$ is constructible. And since $\mathcal{T}_1$ is contained in the extension
	\begin{equation*}
		\mathbb{Q}(e^i,e^{i\sqrt{2}},e^{i\sqrt{3}},e^{i\sqrt{5}}\dots) \subset \mathcal{T}
	\end{equation*}
	Thus the transcendence degree of $\mathcal{T}_1$ is infinite.
\end{proof}
\begin{theorem}\label{thm:exp1}
	For $a,b\in \mathcal{T}_1$ and $\beta \in \mathbb{R}$, then $\alpha^\beta$ is constructible if and only if $|\alpha|^\beta$ is constructible. 
\end{theorem}
\begin{proof}
	One direction is simple since for real $\beta$ then we can take the absolute value of $\alpha^\beta$ to give us $|\alpha^\beta|=|\alpha|^\beta$. For the other direction notice that for any trigonometric $\alpha$ we can break it down into $\alpha = |\alpha|e^{i\delta}$. Like so
	\begin{align*}
		\alpha&=|\alpha|e^{i\delta}\\
		\alpha^\beta &=\left(|\alpha|e^{i\delta}\right)^\beta\\
		\alpha^\beta &=|\alpha|^\beta e^{i\delta\beta}\\
	\end{align*}
	Since $e^{i\delta\beta}$ is trigonometric then if $|\alpha|^\beta$ is trigonometric then $\alpha^\beta$ is trigonometric.
\end{proof}
\begin{lemma}
	For $\alpha\in S_1$ and trigonometric then $\sqrt[n]{\alpha}$ is trigonometric. (Setting $\alpha =1$ gives you the roots of unity.)
\end{lemma}
\begin{proof}
	The special case of Theorem \vref{thm:exp1} where $|\alpha|=1$ since $1^\beta=1$. Any nth root of unity is constructible since $e^{\frac{2\pi i}{n}}=\zeta_n$.
\end{proof}
For the next result, we need a quick definition:
\begin{definition}[Abelian Group]
	A group is \textbf{abelian}, if for any 2 elements $a,b$ then $a \star b = b \star a$.
\end{definition}
\begin{theorem}[Kronecker-Weber]
	If $\mathcal{F}$ is a field with a finite abelian Galois extension then $\mathcal{F}$ is contained in the field generated by some $n$th root of unity.
	\label{thm:kron}
\end{theorem}
Due to Kronecker and Weber the author originally found the theorem in \textcite{dummit_abstract_2009} on page 600.
\begin{corollary}
	Every number inside a finite abelian extension of $\mathbb{Q}$ is contained in $\mathcal{T}_1$. 
\end{corollary}
This also with a quick lemma lets us take care of showing that our algebraic definition of $\mathcal{T}_1$ also contains all constructible numbers
\begin{lemma}
	Every constructible number is contained in $\mathcal{T}_1$, alternatively written as $\mathcal{C} \subseteq \mathcal{T}_1$
	\label{lem:t1constructible}
\end{lemma}
\begin{proof}
	Every quadratic has degree $2$ and thus has galois group of at most degree $2!=2$, since every group of order $2$ is abelian it is contained in $\mathcal{T}_1$
\end{proof}
This shows that there is a vast array of polynomials that can be created through this process, including the roots of irreducible polynomials to any degree.
\section*{Roots of polynomials with nonabelian galois groups in $\mathcal{T}_1$ and Doubling the Cube is not possible with Angle Partitons}
The main goal of this section will be to show that it is possible to construct algebraic numbers outside of the abelian numbers generated by the Kronecker-Weber theorem. And build the foundations for why certain numbers such as $\sqrt[3]{2}$ are probably not contained in $\mathcal{T}_1$

\begin{lemma}
	For any nontrivial Pythagorean triple $a^2+b^2=c^2$  Then
	\begin{equation*}
		\frac{a+bi}{c}
	\end{equation*}
	is not a root of unity.
\end{lemma}
\begin{proof}
	Via the previous theorem, all we need to do is check that $a+bi$ is not a root of unity of degree $n$ where $\varphi(n)=2$, namely we have to check $3$, $4$ and $6$. And since the 3rd and 6th roots of unity involve Gaussian integers divided by square roots. And the fact that $a$ and $b$ have to be not equal to zero takes care of the 4th roots of unity.
\end{proof}
\begin{lemma}\label{lem:minpoly1}
	Minimal polynomial of the trigonometric number $\sqrt[n]{\frac{a+bi}{\sqrt{a^2+b^2}}}$ for $a,b \in \mathbb{R}$ is $x^{2n} -(2a(a^2+b^2)^{-3/2})x^n+1$
\end{lemma}
\begin{proof}
\begin{align*}
	\sqrt[n]{\frac{a+bi}{\sqrt{a^2+b^2}}} &= x \\
	a+b\sqrt{-1} &= x^n\sqrt{a^2+b^2}  \\
	b\sqrt{-1} &=x^n\sqrt{a^2+b^2}-a \\
	-b^2&=(x^n\sqrt{a^2+b^2} -a)^2 \\
	0&=(a^2+b^2)x^{2n} -2ax^n\sqrt{a^2+b^2}+a^2+b^2\\
	0&=x^{2n}-\frac{2a}{\sqrt{a^2+b^2}}x^n+1
\end{align*}
\end{proof}
\begin{theorem}
	$\mathcal{T}_1$ contains algebraic numbers with nonabelian field extensions.
\end{theorem}
\begin{proof}
	Consider the number $\frac{3+4i}{5}$, via Theorem \vref{thm:exp1}, the number $\sqrt[3]{\frac{3+4i}{5}}$ is trigonometric. Since it's algebraic it must have a minimal polynomial and using lemma \vref{lem:minpoly1} it has a minimal polynomial of $5x^6 - 6x^3 + 5$. The Galois group of this subgroup of the permutation group $S_6$ is generated by the following elements.
	\begin{equation*}
		\left\{(24),(35),(12),(34),(56)\right\}
	\end{equation*}
	This group is not abelian since $(24)(12)\neq (12)(24)$.\footnote{Calculated using the Magma CAS}
\end{proof}
We can solve many more polynomials with nonabelian galois extensions.
\begin{theorem}
	The solutions for a reduced cubic polynomial $x^3+px+q=0$ are expressible using trigonometric functions:
	\begin{equation*}
		x_k=2 \sqrt{-\frac{p}{3}} \cos \left[\frac{1}{3} \arccos \left(\frac{3 q}{2 p} \sqrt{\frac{-3}{p}}\right)-\frac{2 \pi k}{3}\right] \quad \text { for } k=0,1,2
	\end{equation*}
\end{theorem}
This quite helpful result is given by \textcite{nickalls_viete_2006} paper on Descartes and the cubic equation. Furthermore doing some results on the cubic discriminant quantity $\frac{3 q}{2 p} \sqrt{\frac{-3}{p}}$. Lets us conclude that the polynomial has 3 real roots if the discriminant has a magnitude less than or equal to 1. And has 2 complex roots and 1 real root otherwise.
\begin{corollary}
	If a cubic has three real roots then it splits over $\mathcal{T}_1$
\end{corollary}
\begin{proof}
	The proof for this hinges on the proof that when the input value to $\arccos$ namely $\left(\frac{3 q}{2 p} \sqrt{\frac{-3}{p}}\right)$ is between $-1$ and $1$, then the polynomial will have three real solutions. The alternative case where a cubic has 1 real and 2 complex roots, (ie $x^3-2$), will involve putting values with a magnitude greater than or equal to 1.
\end{proof}
Does this describe all the polynomials we can create with our algebra? This and other questions have a nonzero probability of potentially being answered in this paper!
\begin{lemma}
	The polynomial $x^{2n}+2cx^n+1$ has solutions in $\mathcal{T}_1$ for trigonometric and real $|c|\leq 1$
\end{lemma}
\begin{proof}
	Consider the minimal polynomial of
	\begin{equation*}
		x=\sqrt[n]{\frac{-1+\sqrt{c^2-1}}{\sqrt{1+c^2-1}}}
	\end{equation*}
	Since it is in our special form we can use Lemma \vref{lem:minpoly1} to say that the minimal polynomial is:
	\begin{align*}
		0&=x^{2n}-\frac{2(-1)}{\sqrt{c^2}}x^n+1\\
		0&=x^{2n}-(-2c)x^n +1\\
		0&=x^{2n}+2cx^n+1
	\end{align*}
	Unfortunately, the limitation is that $\sqrt{c^2-1}$ must be real forces $c$ to be in $0\leq c \leq 1$. Luckily doing some variable substitution and flipping the sign of $a$ in our initial number gives us the minimal polynomial
	\begin{gather*}
		x=\sqrt[n]{\frac{1+\sqrt{c^2-1}}{\sqrt{1+c^2-1}}}\\
		0=x^{2n}-2cx^n+1
	\end{gather*}
	Thus letting us conclude that $|c|\leq 1$.
\end{proof}
\begin{lemma}
	Given a root $z$ of $x^{2n}+2cx^n+1$, the polynomial splits over $\{z,\bar{z},z\zeta_n^1,\bar{z}\zeta_n^1,\dots,z\zeta_n^n,\bar{z}\zeta_n^n\}$
\end{lemma}
\begin{proof}
	The step is to verify that if $z$ is a root of $x^{2n}+2cx^n+1$ then so is $z\zeta^k_n$ we can do this with substitution and the definition that $(\zeta_n^k)^n=1$ we can simplify down
	\begin{align*}
		0&=(z\zeta_n^k)^{2n}+c(z\zeta_n^k)^{n}+1\\
		0&=z^n(\zeta_n^k)^{2n}+cz^n(\zeta_n^k)^{n}+1\\
		0&=z^n(\zeta_n^k)^{2n}+cz^n(\zeta_n^k)^{n}+1\\
		0&=z^n+cz^n+1\\
		0&=0
	\end{align*}
	Thus verifying that $z\zeta^k_n$ is a root. Since the polynomial has all real coefficients the conjugate of every root $\overline{z \zeta^k_n}$ is also a root. Since the conjugate of any root of unity is also a root of unity, you can write the set of $2n$ roots as $\{z \zeta^k_n,\bar{z}\zeta^k_n\}$
\end{proof}

\begin{theorem}
	If $p(x)$ is an irreducible polynomial over the rationals with real and complex roots then it has no roots in $\mathcal{P}$
	\label{thm:thebigone}
\end{theorem}
\begin{proof}
	Consider what would happen if a real root $k_R$ of $p(x)$ was included, since $\mathcal{P}(x) \cap \mathcal{R}$ is splitting then every root of $p(x)$ must be in $\mathcal{P}(x) \cap \mathcal{R}$, this is impossible since some roots of $p(x)$ are complex. Likewise any complex root $k_I$ cannot be included since $\mathcal{P}$ is splitting, but $k_R$ cannot be included so neither can $k_I$.
\end{proof}
The first form I have seen of a proof for this in the literature is \cite{gleason_angle_1988}, where he proves a weaker version for cubics and angle trisection only using Cardano's closed-form solution. (There is an error in this earlier formulation where he states: "A real cubic equation can be solved geometrically using ruler, compass, and angle-trisector if and only if its roots are all real", this is true of irreducible cubics, but simple examples like $x^3-1$ disprove the general case.) 
\begin{corollary}
	Every irreducible polynomial with an abelian galois extension has all real or all complex roots.
\end{corollary}
\begin{proof}
	Via Theorem \vref{thm:kron} our polynomial must have all its roots constructible using angle partitions. Thus our polynomial must have all real or all complex roots via Theorem \vref{thm:thebigone}.
\end{proof}
\begin{corollary}
	Any odd $k$th root of a prime number $\sqrt[k]{p}$, is not constructible using angle partitions, for any odd $k$ greater than 3.
\end{corollary}
\begin{proof}
	Since the minimal polynomial $x^n-p$ is irreducible via Eisenstein's criterion since $p$ divides every coefficient except for the first and $p^2$ does not divide $p$. For $n \geq 3$ the polynomial has both real and complex roots then $\sqrt[n]{p}$ is not constructible using partitions. (Also serves as an argument that doubling the cube is impossible with conventional euclidean geometry.)
\end{proof}

This leads us to the following result that will be further expanded upon in the conclusion.
\begin{conjecture}
	Since doubling the cube is impossible with angle partitions it seems unlikely that it is possible in $\mathcal{T}_1$.
\end{conjecture}
\section*{General Intersections of Transcendental Curves}
Limiting ourselves to certain intersections let us use the similar geometric interactions of all curves to produce a field with an interesting algebraic structure. Letting in other kinds of intersections introduces a great amount of complexity to the problem.  Consider first the field created by allowing any sloped line intersections with simple trigonometric functions like $\sin$ and $\cos$

\begin{definition}
	The field generated by allowing sloped line intersections with $\sin$ or $\cos$ is the field $\mathcal{T}_1$ adjoined with the roots of the expression $\sin{x}+ax+b$
\end{definition}
\begin{theorem}\label{thm:t2tran1}
	If both $a$ and $b$ are algebraic then any solution $x$, q$\sin^{-1}(x)$ are transcendental.
\end{theorem}
\begin{proof}
	First, assume that $x$ is algebraic, then since $a$ and $b$ are algebraic then $ax+b$ is algebraic, and therefore $\sin{x}$ is algebraic, but if $x$ is algebraic then $\sin{x}$ is algebraic, but via [Lindemann–Weierstrass] $\sin{x}$ and $x$ cant both be algebraic so by contradiction $x$ is transcendental. Since $ax+b$ is transcendental then $\sin{x}$ is also transcendental.
\end{proof}
\begin{corollary}
	Under the constraints of Theorem \vref{thm:t2tran1} then for any solution $x$ then $\sin^{-1}(x)$ must be transcendental
\end{corollary}
\begin{proof}
	Because $\sin{\pi\cdot \frac{p}{q}}$ is algebraic, it contradicts the aforementioned theorem. Thus $\pi\cdot \frac{p}{q}$ is not a valid solution.
\end{proof}

\begin{corollary}
	For any numbers $a \in \mathbb{R}$ and $b \in \mathcal{T}_1$ the equation $\sin(x)=ax+(\sin(b)-ab)$ has a trigonometric solution
\end{corollary}
\begin{proof}
	A solution of $x=b$ satisfies the equation.
\end{proof}
\begin{corollary}
	The problem statement is equivalent to finding all roots of $bx-\frac{x^3}{6}+\frac{x^5}{25}+\dots$
\end{corollary}
\begin{proof}
	Using the Taylor series expansion for $\sin(x)$ we have that we are trying to find the roots of
	\begin{equation*}
		-\left(b_1x+a_1\right)+\left(x-\frac{x^3}{3!}+\frac{x^5}{5!}-\frac{x^7}{7!}\right)
	\end{equation*}
\end{proof}
\begin{theorem}
	Any sloped line intersecting the quadratrix involves finding the solution to $a+\frac{b}{x}=\cot(x)$
\end{theorem}
\begin{proof}
	\begin{align*}
		ax+b&=x\cot(\pi x/2)\\
		a+\frac{b}{x}&=\cot(\pi x /2) \\
		a+\frac{b}{x}&=\cot(x)
	\end{align*}
\end{proof}

Furthermore, some basic plots seem to indicate that there is a trove of complex solutions to even the simple base case of $\sin(x)-ax$
\begin{figure}[h!]
	\center
	\includegraphics[width=.6\textwidth]{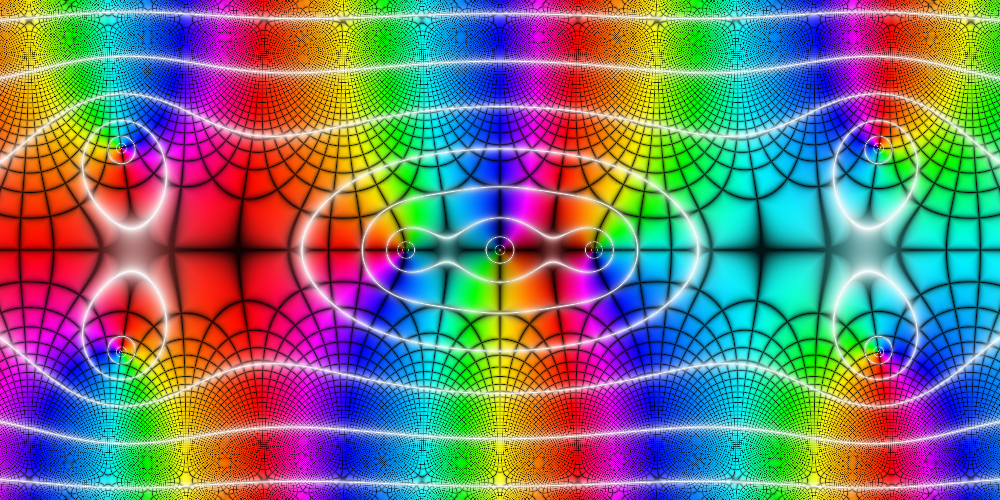}
	\caption[]{$\sin(x)-x/2$ in the complex plane}
	\label{fig:sinhalf} 
\end{figure}
Furthermore, the where $\sin(x)$ and $ax$ intersect and are tangent has a fascinating expression in the complex plane, where the complex solutions combine on the real line into a solution with a multiplicity of 2, before separating on the real line as seen in Figure \vref{fig:sintangent}.
\begin{figure}[h!]
	\center
	\includegraphics[width=.6\textwidth]{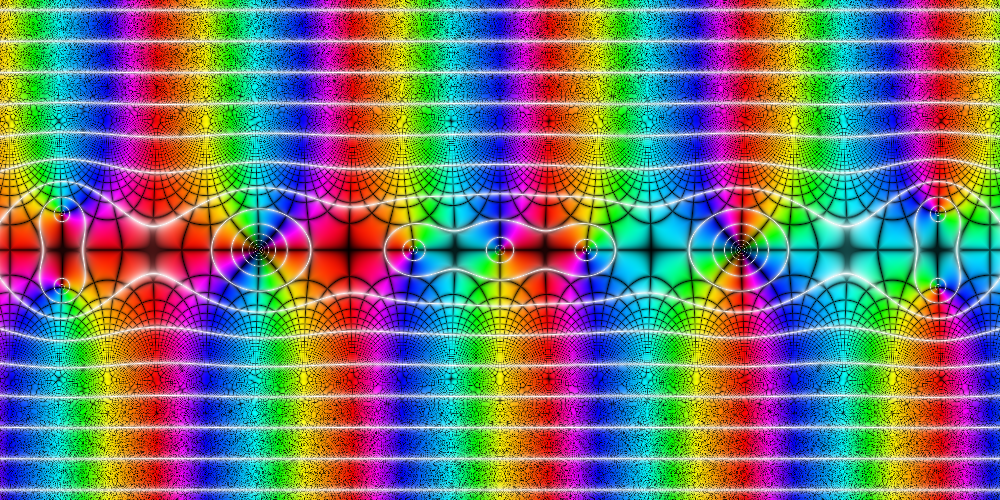}
	\caption[]{$\sin(x)-0.1283x$ in the complex plane}
	\label{fig:sintangent} 
\end{figure}
What's interesting is that the characteristic equations for the quadratrix seem to buck the trend and only have solutions appearing on the real line.
\begin{figure}[h!]
	\center
	\includegraphics[width=.6\textwidth]{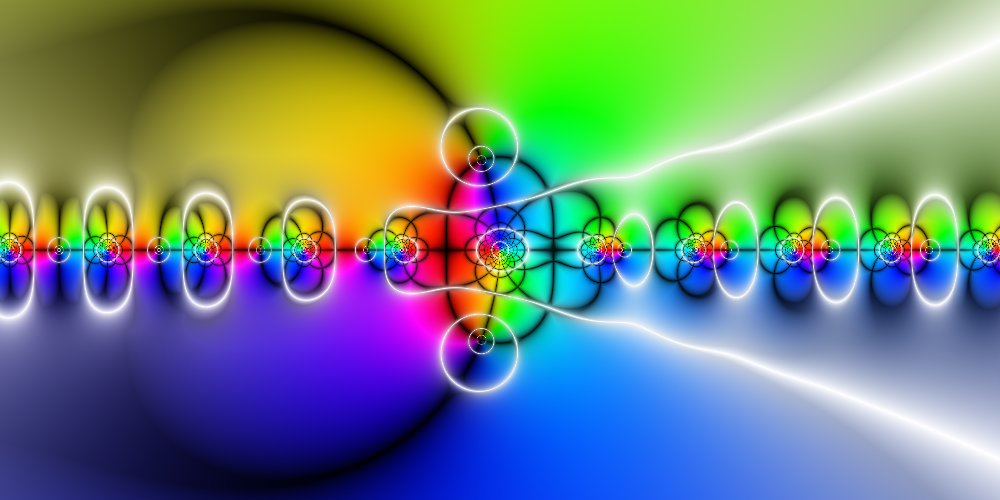}
	\caption[]{$\cot(x)-.2-3/x$ in the complex plane}
	\label{fig:quad1} 
\end{figure}

\section*{Conclusion}
Despite the best efforts of the author some important questions on this matter remain unresolved, 
\begin{conjecture}
	Does the field created by allowing arbitrary partition of angles $\mathcal{P}$ completely encapsulate the algebraic numbers in $\mathcal{T}_1$
\end{conjecture}
The strategy used for this involved turning our geometric constructions into algebra and then using our base algebraic objects to construct the following class of functions that take in algebraic numbers and output algebraic numbers.
\begin{gather*}
	f_q: \mathbb{A} \rightarrow \mathbb{A} \\
	f_q(z)= e^{q \cdot \ln(z)}
\end{gather*}
However, saying that this class of functions generates all algebraic numbers would involve proving that there are no other "ways" of making algebraic numbers. For example, consider the function

\begin{gather*}
	g: \mathbb{A}\times\mathbb{A}\times\mathbb{A} \rightarrow \mathbb{C}\\
	g(a,b,c)=e^{\frac{\ln(a)\ln(b)}{\ln(c)}}
\end{gather*}
We know that this function does produce algebraic output when $c$ equals 0. Or when either $a$ or $b$ are expressible as $c$ raised to a rational power. (Actually plugging in values gives you results that look like $2^{\log_3(5)}$.) There might very well be a theorem in transcendental number theory that could easily resolve this.

Likewise, there are a bunch of additional questions that could be answered about allowing general intersections, namely.
\begin{conjecture}
	Are all the fields allowing general intersections with the transcendental curves (trig functions, Quadratrix of Hippias, Archemedian Spiral) identical? (Do they form a larger $\mathcal{T}_2$)
\end{conjecture}
I suspect that if it is true it would be very hard to prove, but on the other hand, it would be very pretty if it was true.
\begin{conjecture}
	Does allowing general intersections in any of the transcendental curves allow the construction of algebraic numbers that are impossible in $\mathcal{T}_1$
\end{conjecture}
There might very well be a clever trick in the algebra to allow this and if any reader would be interested in continuing this work further I would recommend they consider starting here.

\section*{Acknowledgements}

I would like to sincerely thank Dr. John Carter for advising me on this project and helping me through the publication process, as well as teaching the History of Mathematics class where the quadratrix was first introduced to me.

\printbibliography

@book{hull_origametry_2020,
	edition = {1},
	title = {Origametry: Mathematical Methods in Paper Folding},
	isbn = {978-1-108-77863-3 978-1-108-47872-4 978-1-108-74611-3},
	url = {https://www.cambridge.org/core/product/identifier/9781108778633/type/book},
	shorttitle = {Origametry},
	publisher = {Cambridge University Press},
	author = {Hull, Thomas C.},
	urldate = {2022-09-28},
	date = {2020-10-08},
	langid = {english},
	doi = {10.1017/9781108778633},
}

@article{nickalls_viete_2006,
	title = {Viète, Descartes and the cubic equation},
	volume = {90},
	issn = {0025-5572, 2056-6328},
	url = {https://www.cambridge.org/core/product/identifier/S0025557200179598/type/journal_article},
	doi = {10.1017/S0025557200179598},
	pages = {203--208},
	number = {518},
	journaltitle = {The Mathematical Gazette},
	shortjournal = {Math. Gaz.},
	author = {Nickalls, R. W. D.},
	urldate = {2022-12-30},
	date = {2006-07},
	langid = {english},
}

@article{gleason_angle_1988,
	title = {Angle Trisection, the Heptagon, and the Triskaidecagon},
	volume = {95},
	issn = {00029890},
	url = {https://www.jstor.org/stable/2323624?origin=crossref},
	doi = {10.2307/2323624},
	pages = {185},
	number = {3},
	journaltitle = {The American Mathematical Monthly},
	shortjournal = {The American Mathematical Monthly},
	author = {Gleason, Andrew M.},
	urldate = {2022-12-30},
	date = {1988-03},
}

@video{polster_why_2019,
	title = {Why don't they teach this simple visual solution? (Lill's method)},
	url = {https://www.youtube.com/watch?v=IUC-8P0zXe8},
	author = {Polster, Burkard},
	date = {2019-04-26},
}

@book{natarajan_pillars_2020,
	location = {Singapore},
	title = {Pillars of Transcendental Number Theory},
	isbn = {9789811541544 9789811541551},
	url = {http://link.springer.com/10.1007/978-981-15-4155-1},
	publisher = {Springer Singapore},
	author = {Natarajan, Saradha and Thangadurai, Ravindranathan},
	urldate = {2022-11-02},
	date = {2020},
	langid = {english},
	doi = {10.1007/978-981-15-4155-1},
}

@book{dummit_abstract_2009,
	location = {New York},
	edition = {3. ed., [Nachdr.]},
	title = {Abstract algebra},
	isbn = {978-0-471-43334-7 978-0-471-45234-8},
	pagetotal = {932},
	publisher = {Wiley},
	author = {Dummit, David Steven and Foote, Richard M.},
	date = {2009},
}

\end{document}